\documentclass[12pt]{article}
\usepackage{bbm}
\usepackage{amssymb}
\usepackage{amsfonts}
\usepackage{amsmath}
\usepackage[usenames]{color}
\usepackage{mathrsfs}
\usepackage{amsfonts}
\usepackage{amssymb,amsmath}
\usepackage{CJK}
\usepackage{cite}
\usepackage{cases}
\usepackage{amsthm}

\def\dis{\displaystyle}
\def\deg{{\rm deg}}

\def\cl{\centerline}

\def\al{\alpha}

\def\vs{\vspace*}

\def\Z{\mathbb{Z}}

\def\C{\mathbb{C}}

\def\V{{W}}

\numberwithin{equation}{section}
\newtheorem{th*}{Theorem}
\newtheorem{theo}{Theorem}[section]

\newtheorem{nota}[theo]{Notation}
\newtheorem{coro}[theo]{Corollary}
\newtheorem{lemm}[theo]{Lemma}

\newtheorem{prop}[theo]{Proposition}
\newtheorem{clai}{Claim}
\newtheorem*{clai*}{Claim}

\newtheorem{remark}[theo]{Remark}

\begin{document}
\begin{center}
{\large\bf  Submodule structures  of    $\mathbb C[s,t]$ over $W(0,b)$
and a new class of irreducible modules over the Virasoro algebra}

\end{center}
\cl{ Jianzhi Han
and Yucai Su}

\vs{8pt}
{
\parskip .01 truein
\baselineskip 5pt \lineskip 5pt

{\small \noindent {\em Abstract}: For any $a,b\in\mathbb C$, $W(a,b)$ is the Lie algebra with basis $\{L_m,M_m\,|\,m\in\mathbb Z\}$ and relations $[L_m,L_n]=(n-m)L_{m+n},$
$[L_m,W_n]=(a+n+bm)W_{m+n}$,
$[W_m,W_n]=0$ for $m,n\in\mathbb Z$.
For any $\lambda\in\mathbb C^*,$ $\alpha\in\mathbb C$, $h:=h(t)\in\mathbb C[t]$,
there exists a non-weight module
over $W(0,b)$ (resp., $W(0,1)$),
 denoted by $\Phi(\lambda,\alpha,h)$  (resp. $\Theta(\lambda,h)$),
which is
defined on the space $\mathbb C[s,t]$ of polynomials on variables $s,t$ and  is
 free of rank one over the enveloping algebra $U(\mathbb C L_0\oplus\mathbb C W_0)$ of $\mathbb C L_0\oplus\mathbb C W_0$. In the present paper, by introducing two sequences of useful operators on $\mathbb C[s,t]$,  we determine all submodules of $\mathbb C[s,t]$. We also study  submodules of $\mathbb C[s,t]$ regarded as modules over the Virasoro algebra $\mathscr V\!$ (with the trivial action of the center), and prove that these submodules are finitely generated  if and only if ${\rm deg}\,h(t)\geq1$.  In addition, it is proven that  $\Phi(\lambda, \alpha,h)$ is an irreducible $\mathscr V\!$-module if and only if  $b=-1$, ${\rm deg}\, h(t)=1$, $\alpha\neq0$.
 Finally, we obtain a large family of new irreducible modules over the Virasoro algebra $\mathscr V\!$, by taking  various tensor products of a finite number of irreducible modules $\Phi(\lambda_i,\alpha_i, h_i)$ for $\lambda_i,\alpha_i\in\mathbb C^*,$ $h_i\in\mathbb C[t]$
with an irreducible $\mathscr V\!$-module  $V$,
  where  $V$ satisfies that  there exists  a nonnegative integer $R_V$ such that $L_m$ acts locally finitely on $V$ for  $m\geq R_V$.

}\parskip .001 truein\baselineskip 6pt \lineskip 6pt

\section{Introduction}
It is well-known that non-weight modules constitute an important ingredient in the representation theory
of Lie algebras. Their study is, in some sense,  more challenging than  that of weight modules.  Recently non-weight modules
over the Virasoro algebra and some related algebras
have been
extensively studied 
 (see, e.g.,  \cite{BM,CG1,CG2,CG3,CHS,CHSY,LiuZ,LZ2,LLZ,MW,MZ,TZ1,TZ2,TZ3}),  especially those
which are  free over the enveloping algebra $U(\mathfrak h)$
of
an abelian subalgebra $\mathfrak h$ contained in some Cartan subalgebra
(see, e.g.,\cite{HCS,N1,N2,TZ3}).

Let us recall the non-weight modules studied in \cite{HCS}.
 For any pair $(a,b)$ of complex numbers, let  $\V(a,b)$ be  the Lie algebra with  basis $\{L_m, W_m\mid m\in\Z\}$ and    the following defining relations, for $m,n\in\Z$
(see, e.g., \cite{OR}),
\begin{eqnarray}\label{def-relation}
&\!\!\!\!\!\!\!\!&[L_m,L_n]=(n-m)L_{m+n},
\ \ \ \ [L_m,W_n]=(a+n+bm)W_{m+n},
\ \ \ [W_m,W_n]=0.
\end{eqnarray}
Note that 
$\V(a,b)$ includes  several  well-known infinite dimensional Lie algebras as subalgebras, such as the Witt algebra  (or the centerless Virasoro algebra) $\mathscr W\!:={\rm span}\{L_m\mid m\in\Z\}$, the centerless twisted Heisenberg-Virasoro algebra $\V(0,0)$, etc.

Fix any $\lambda\in\C^*, \al\in\C$, $h:=h(t)\in\C[t]$ (the polynomial algebra in the variable $t$ with coefficients in $\C$).  We define the following elements in  $\C[t]$, for $m\in\Z$,
\begin{eqnarray}
\label{g-hi-}
&\!\!\!\!\!\!\!\!&
g(t)=\frac{h(t)-h(\al)}{t-\al},
\nonumber\\
&\!\!\!\!\!\!\!\!&
 h_m(t)=mh(t)-m\al\Big(\delta_{b,-1}(m-1)+\delta_{b1}\Big)g(t).
\end{eqnarray}
We define the action of $W(0,b)$ on the  space $\C[s,t]$ of polynomials on variables $s,t$, which is denoted as $\Phi(\lambda,\al, h)$,
in the following way,  for $f(s,t)\in\C[s,t]$,
\begin{eqnarray}
\label{phi-module}
&\!\!\!\!\!\!\!\!\!\!\!\!\!\!\!\!&L_mf(s,t)=\lambda^m\Big(\big(s+h_m (t)\big)f(s-m,t)+bm(t-\delta_{b,-1}m\al-\delta_{b1}\al)\displaystyle \frac{\partial}{\partial t}f(s-m,t)\Big),\nonumber\\
&\!\!\!\!\!\!\!\!\!\!\!\!\!\!\!\!&W_m f(s,t)=
\lambda^m\Big(t-\delta_{b,-1}m\al-\delta_{b1}(1-\delta_{m0})\al\Big)f(s-m,t).
\end{eqnarray}
We also define the action of $W(0,1)$ on the  space $\C[s,t]$, which is now denoted as $\Theta(\lambda, h)$,
in the following way,   for $f(s,t)\in\C[s,t]$,
\begin{eqnarray}
\label{theta-module}
&\!\!\!\!\!\!\!\!\!\!\!\!\!\!\!\!&
L_m f(s,t)=\lambda^m\Big(s+mh(t)\Big)f(s-m,t),\nonumber\\
&\!\!\!\!\!\!\!\!\!\!\!\!\!\!\!\!&
W_m f(s,t)=\delta_{m0}tf(s,t).
\end{eqnarray}
The following is obtained in \cite{HCS}.
\begin{th*} \label{theo1}
\begin{itemize}
\item[\rm (1)]  $\Phi(\lambda,\al,h)$ is a $\V(0,b)$-module$,$ and
 $\Theta(\lambda,h)$ is a $\V(0,1)$-module$.$

\item[\rm (2)]  Assume $M$ is  a $\V (a,b)$-module  which is  free of rank one over the enveloping algebra $U(\C L_0 \oplus\C W_0)$. Then $a=0$ and $M\cong \Phi(\lambda,\al,h)$ or $\Theta(\lambda,h)$ if $b=1,$ or
 $M\cong\Phi(\lambda,\al,h)$  if $b\neq 1,$ for some $\alpha\in\C, \lambda\in \C^*,$ $h\in\C[t].$
 \end{itemize}
\end{th*}

Note from Theorem \ref{theo1}\,(2) that if $a\ne0$, there does not exist a
 $\V (a,b)$-module which is free of rank one over  $U(\C L_0 \oplus\C W_0)$. Therefore in the present paper, we will focus on the study of submodules of
 $\V(0,b)$-module $\Phi(\lambda, \alpha,h)$ and $\V(0,1)$-module $\Theta(\lambda,h)$.
 It is well-known that the Virasoro algebra
 $\mathscr V\!$ is the universal central extension of the Witt algebra $\mathscr W\!$, i.e., $\mathscr V\!$ is the Lie algebra with basis $\{L_m,C\,|\,m\in\Z\}$ and the following relations, for $m,n\in\Z$,
\begin{eqnarray}
\label{vir}
[L_m,L_n]=(n-m)L_{m+n}+\delta_{m+n,0}\frac{m^3-m}{12}C,\ \ \ \ [L_m,C]=0.
\end{eqnarray}
Then
  $\Phi(\lambda,\alpha,h)$, $\Theta(\lambda,h)$ are naturally non-weight $\mathscr V\!$-modules with the trivial action of $C$.
By introducing two sequences of operators $S^j$ and  $T^j$ for $j\geq 0$ (cf.~\eqref{operator-ST}), we are able to determine
when  $\Phi(\lambda,\alpha,h)$, $\Theta(\lambda,h)$ are   irreducible modules over $\mathscr V\!$, and furthermore, by taking tensor products of these irreducible modules with  $\mathscr V\!$-modules studied in \cite{MZ},
we obtain a large class of new irreducible non-weight $\mathscr V\!$-modules.
We would like to emphasis that the operators $S^j,\,T^j$ play
 very important roles in the paper: they not only allow us to
determine whether or not a subspace of $\C[s,t]$ is a submodule over $\V(0,b)$ (or $\mathscr V$), but also allow us to compute cyclic submodules over $\V(0,b)$ (or $\mathscr V$).

The rest of the present paper is organized as follows.  Section 2 is devoted to determining $\V(0,b)$-submodules of $\Phi(\lambda,\alpha, h)$ and $\V(0,1)$-submodules of $\Theta(\lambda,h).$
The determinations are divided into three subsections according to whether $b=0,\,b=1$ or $b\ne0,1$.
%
In Section 3, all cyclic $\mathscr V\!$-submodules of $\C[s,t]$ are determined. We also
prove that  all $\mathscr V\!$-submodules  of $\C[s,t]$ are finitely generated.
Furthermore, we obtain that the $\mathscr V\!$-module  $\Theta(\lambda,h)$ is always reducible and that
the  $\mathscr V\!$-module $\Phi(\lambda,\alpha,h)$ is irreducible if and only if $\alpha\in\C^*, {\rm deg}\, h(t)=1$ and $b=-1$. In Section 4, we obtain a large family of new irreducible modules over the Virasoro algebra $\mathscr V\!$ by taking various tensor products of a finite number of irreducible modules $\Phi(\lambda_i,\alpha_i, h_i)$ for $\lambda_i\in\C^*,$ $\alpha_i\in\C$, $h_i\in\C[t]$
with an irreducible $\mathscr V\!$-module  $V$,
  where  $V$ satisfies that  there exists $R_V\in\Z_+$ such that $L_m$ acts locally finitely on $V$ for  $m\geq R_V$.

Throughout the paper,  $\Z_+, \C, \C^*$ denote the sets of nonnegative integers,  complex numbers and nonzero complex numbers, respectively.

\section{$\V(0,b)$-submodules}
 The aim of this section is to investigate reducibility of  $\V(0,b)$-modules  $\Phi(\lambda,\alpha,h)$ and  $\V(0,1)$-modules $\Theta(\lambda,h)$, and determine all submodules when they are reducible.

Let $\lambda\in\C^*, \al\in\C$, $h(t)\in\C[t]$ and $g(t)\in\C(t)$ be as in
\eqref{g-hi-}.
For $j\in\Z$ we denote $\partial_s^j=0$ if $j<0$ and
$\partial_s^j=\big(\frac{\partial}{\partial s}\big)^j$ if $j\ge0$, and set $\partial_t=\frac{\partial}{\partial t}$.
Define the following operators on $\C[s,t]$, for $j\in\Z_+$,
\begin{eqnarray}
\label{operator-ST}
&\!\!\!\!\!\!\!\!\!\!\!\!\!\!\!\!\!\!\!\!&
T^j=\frac {t-\delta_{b1}\alpha}{j!}\partial_s^j+\frac{\delta_{b,-1} \alpha}{(j-1)!}\partial_s^{j-1},
\ \ \ \ \ \ \ 
 S^j_{\Theta}=\frac s{j!}\partial_s^j-\frac {h(t)}{(j-1)!}\partial_s^{j-1},\nonumber\\[4pt]
&\!\!\!\!\!\!\!\!\!\!\!\!\!\!\!\!\!\!\!\!&
S^j=\frac s{j!}\partial_s^j-\frac 1{(j-1)!}\Big(h(t)+\delta_{b,-1}\alpha g(t)-\delta_{b1}\alpha g(t)+(bt-\delta_{b1}\alpha)\partial_t\Big)\partial_s^{j-1}\nonumber\\
&\!\!\!\!\!\!\!\!\!\!\!\!\!\!\!\!\!\!\!\!&
\phantom{S^j=}-\frac{\delta_{b,-1}\alpha }{(j-2)!}\Big(g(t)-\partial_t\Big)\partial_s^{j-2},
\end{eqnarray}
where we have used  the convention that  $k!=1$ if $k<0$. In the following we will also use the convention that
 $\binom{i}{j}=0$  if  $j> i$ or $j<0$.

\begin{lemm}\label{lemm-1}
Let $M$ be a subspace of $\C[s,t]$. Then
\begin{itemize}\item[\rm(1)]
$M$ is a $\V(0,b)$-submodule of $\Phi(\lambda,\alpha,h)$  if and only if $M$ is invariant under the actions of $S^j,$ $T^j$ for $j\in\Z_+;$
\item[\rm(2)]$M$ is a $\V(0,1)$-submodule of $\Theta(\lambda, h)$ if and only if $M$ is invariant under
the left multiplication of $\,t$ and under the actions of
 $S^j_{\Theta}$ for $j\in\Z_+$.
\end{itemize}
In particular$,$ $M$ is a $\mathscr W\!$-submodule of $\Phi(\lambda,\alpha,h)$ $($resp. $\Theta(\lambda,h)$$)$ if and only if $M$ is invariant under the operators $S^j$ $($resp. $S^j_\Theta$$)$ for $j\in\Z_+$.
\end{lemm}
\begin{proof}We will only prove the statements concerning $\Phi(\lambda, \alpha, h)$ as the proof for the statements concerning  $\Theta(\lambda, h)$ is similar.  For any $f(s,t)=\sum_{i=0}^ns^if_i(t)\in\C[s,t]$ with $f_i(t)\in\C[t]$ and $n\in\Z_+$, by \eqref{phi-module} and \eqref{operator-ST}, we have
\begin{eqnarray*}\
W_mf(s,t)&\!\!\!=\!\!\!& \lambda^m\Big(t-\delta_{b,-1}m\al-\delta_{b1}(1-\delta_{m0})\al\Big)\mbox{$\sum\limits_{i=0}^n$}(s-m)^{i}f_i(t)\\
&\!\!\!=\!\!\!& \lambda^m\mbox{$\sum\limits_{j=0}^{n+1}$}(-m)^j\mbox{$\sum\limits_{i=0}^n$}s^{i-j}\mbox{\LARGE$\Big($}\!
\binom{i}{j}\Big(t-\delta_{b1}(1-\delta_{m0})\al\Big)+\binom{i}{j-1}\delta_{b,-1}\alpha s\!\mbox{\LARGE$\Big)$}f_i(t)\\
&\!\!\!=\!\!\!&  \lambda^m\mbox{$\sum\limits_{j=0}^{n+1}$} (-m)^j(T^j+\delta_{b1}\delta_{m0}\alpha)f(s,t),
\\[6pt]
L_mf(s,t)&\!\!\!=\!\!\!& \lambda^m\Big(s+h_m (t)\Big)f(s-m,t)+\lambda^mbm\big(t-\delta_{b,-1}m\al-\delta_{b1}\al\big)\partial_t\Big(f(s-m,t)\Big)\\
&\!\!\!=\!\!\!& \lambda^m
\mbox{\large$\Big($}\!
-m^2\delta_{b,-1}\alpha \Big(g(t)-\partial_t\Big)+m\Big(h(t)+\delta_{b,-1}\alpha g(t)-\delta_{b1}\alpha g(t)\\ &&\quad \ \ \  \ +\,(bt-\delta_{b1}\alpha)\partial_t\Big)+s\! \mbox{\large$\Big)$}\mbox{$\sum\limits_{i=0}^n$} (s-m)^if_i(t)\\
&\!\!\!=\!\!\!& \lambda^m\mbox{$\sum\limits_{j=0}^{n+2}$}(-m)^j\mbox{$\sum\limits_{i=0}^n$}
s^{i+1-j}
\mbox{\LARGE$\Big($}\!
\binom{i}{j}-\binom{i}{j-1}\Big(h(t)+\delta_{b,-1}\alpha g(t)-\delta_{b1}\alpha g(t)\\
&&\phantom{%
\lambda^m\mbox{$\sum\limits_{j=0}^{n+2}$}(-m)^j\mbox{$\sum\limits_{i=0}^n$}
s^{i+1-j}
\mbox{\LARGE$\Big($}\!
}
+(bt-\delta_{b1}\alpha)\partial_t\Big)-\binom{i}{j-2}\delta_{b,-1}\alpha s(g(t)-\partial_t)
\!\mbox{\LARGE$\Big)$}
f_i(t)\\
&\!\!\!=\!\!\!& \lambda^m \mbox{$\sum\limits_{j=0}^{n+2}$} (-m)^jS^jf(s,t).
\end{eqnarray*}
The above imply the following (note from definition \eqref{operator-ST} that when applying the following to any element in $\C[s,t]$ there are only finite nonzero terms)
\begin{equation}\label{ammel}L_m=\lambda^m\mbox{$\sum\limits_{j=0}^\infty$}(-m)^jS^j,\quad {\rm }\quad W_m=\lambda^m\mbox{$\sum\limits_{j=0}^\infty$}(-m)^j (T^j+\delta_{b1}\delta_{m,0}\alpha).\end{equation} Therefore, a subspace $M$ of $\C[s,t]$ is a $\V(0,b)$-submodule if and only if $S^jf(s,t), T^jf(s,t)\in M$ for all  $f(s,t)\in M$ and $j\in\Z_+.$
\end{proof}

For any $f(s,t)\in\C[s,t]$,  we use $\mathcal S_{f(s,t)}$ to denote the $\V(0,b)$-submodule of $\Phi(\lambda,\alpha,h)$, or the $\V(0,1)$-submodule of $\Theta(\lambda,h)$, generated by $f(s,t)$.
Now we divide the determinations of submodules into three subsections according to whether $b=0$, $b=1$ or $b\ne0,1$.
%
%
%

\subsection{The $\V(0,0)$-submodules of $\Phi(\lambda,\alpha,h)$}
In this subsection we assume  $b=0$.
Then $T^j, S^j$ in \eqref{operator-ST} have the following simplified forms,
\begin{eqnarray*}
T^j=\frac t{j!}\partial_s^j,\ \ \ \ \ \ \ \ {\rm }\ S^j=\frac s{j!}\partial_s^j-\frac 1{(j-1)!}\partial_s^{j-1}h(t).
\end{eqnarray*}
From this, it is easy to check that for $f(t)\in\C[t]$, we have
\begin{eqnarray*} \mathcal S_{f(t)}=\C[s,t]f(t),\ \ \ \quad
\mathcal S_{sf(t)}=\C[s,t]sf(t)+\C[s,t]tf(t)+\C[s,t]h(t)f(t).
\end{eqnarray*}
In particular,
$\mathcal S_{sf(t)}=\mathcal S_{f(t)}$ if $h(0)\neq 0$, and
  $\mathcal S_{sf(t)}=\C[s,t]sf(t)+\C[s,t]tf(t)$ if $ h(0)=0$.

\begin{lemm}\label{lemm-2(b=0)}Let $f(s,t)=\sum_{i=0}^ns^if_i(t)\in\C[s,t]$ for some $f_i(t)\in\C[t]$ and $1\le n\in\Z_+$.
Then \begin{equation}
\label{eqn-lemm-2(b=0)}
\mathcal S_{f(s,t)}=\mbox{$\sum\limits_{i=1}^n$}\mathcal S_{sf_i}+\mathcal S_{f_0}.
\end{equation}
\end{lemm}
\begin{proof}
The right-hand side of \eqref{eqn-lemm-2(b=0)} is obviously a $\V(0,0)$-submodule of $\Phi(\lambda,\alpha,h)$  containing $f(s,t)$. Thus,  it suffices to prove that $\sum_{i=1}^n\mathcal S_{sf_i}+\mathcal S_{f_0}\subseteq \mathcal S_{f(s,t)}$, which, by induction on $n\ge0$, can be reduced to proving that $\mathcal S_{sf_n}\subseteq \mathcal S_{f(s,t)}$. The later is  equivalent to that $sf_n(t)\in \mathcal S_{f(s,t)}$.

First assume $h(0)\neq0.$
Note from Lemma \ref{lemm-1} that $sh(t)f_n(t)=-S^0S^{n+1}f(s,t)\in \mathcal S_{f(s,t)}$ and that  $$s\Big(h(t)-h(0)\Big)f_n(t)= \frac{S^0\Big(h(T^0)-h(0)\Big)T^nf( s, t)}{t}\in\mathcal S_{f(s,t)} .$$ Thus,  $sf_n(t)\in \mathcal S_{f(s,t)}$.
Now assume
$h(0)=0,$ i.e., $t\,|\,h(t)$.
 It follows from Lemma \ref{lemm-1} that  $stf_n(t)=S^0T^nf(s,t)\in\mathcal S_{f(s,t)}$ and $nstf_n(t)+tf_{n-1}(t)=T^{n-1}f(s,t)\in \mathcal S_{f(s,t)},$ which imply that $tf_{n-1}(t)\in \mathcal S_{f(s,t)}$. Then by our assumption that $t\, |\, h(t)$ and by Lemma \ref{lemm-1} again, we obtain $$h(t)f_{n-1}(t)=h(T^0)f_{n-1}(t)\in \mathcal S_{f(s,t)}.$$ These and the fact that $sf_n(t)-nsh(t)f_n(t)-h(t)f_{n-1}(t)=S^nf(s,t)\in \mathcal S_{f(s,t)}$  immediately imply that $sf_n(t)\in\mathcal S_{f(s,t)}$.
\end{proof}

\begin{theo}\label{theo-1(b=0)}
 The set $\big\{\mathcal S_{sf}, \mathcal S_f\mid f:=f(t)\in \C[t]\big\}$
 exhausts all  $\V(0,0)$-submodules of $\Phi(\lambda, \alpha, h)$. In particular$,$  all $\V(0,0)$-submodules of $\Phi(\lambda,$ $\alpha,h)$ are cyclic.
 \end{theo}

\begin{proof}
Let $M$ be a nonzero proper $\V(0,0)$-submodule of $\Phi(\lambda,\alpha,h)$. Note from Lemma \ref{lemm-2(b=0)} that there exist nonzero polynomials $f:=f(t),\,g:= g(t)\in\C[t]$ such that $sf(t), g(t)\in M$.
Without loss of generality, we may assume both ${\rm deg}\, f(t)$ and ${\rm deg}\, g(t)$ are minimal (we regard the degree of the zero polynomial as $\infty$).

\begin{clai*}
$M=\mathcal S_{sf}+\mathcal S_g.$
\end{clai*}
It suffices to prove that $M\subseteq\mathcal S_{sf}+\mathcal S_{g}$. For any $p(s,t)=\sum_{i\geq0} s^ip_i(t)\in M,$ we know from Lemma \ref{lemm-2(b=0)} that $sp_i(t), p_0(t)\in M$ for $i\ge1$. Now by the minimalities of ${\rm deg}\, f(t)$ and ${\rm deg}\, g(t)$, we must have $f(t)\mid p_i(t)$ and $g(t)\mid p_0(t)$ for $i\geq1$.  Thus, $p(s,t)\in \mathcal S_{sf}+\mathcal S_g$, proving the claim.

Since $sg(t), tf(t)\in M,$ by the minimalities of deg $f(t)$ and deg $g(t)$ we have $f(t)\mid g(t)$ and $g(t)\mid tf(t)$. Thus, $g(t)=f(t)$ or $g(t)=tf(t)$ (up to a nonzero scalar), and accordingly,  $M=\mathcal S_g$ or $M=\mathcal S_{sf}$.
\end{proof}

Now  
we  describe all maximal $\V(0,0)$-submodules of the cyclic submodule $\mathcal S_{f(s,t)}$ of $\Phi(\lambda,\alpha,h)$ and determine when $\mathcal S_{f(s,t)}=\mathcal S_{sg}$ or $\mathcal S_{f(s,t)}=\mathcal S_g$ for some $g\in\C[t]$.

\begin{prop}\label{prop-1(b=0)}
\begin{itemize}
\item[\rm(1)] Let $0\neq f(t)\in\C[t]$. Every maximal $\V(0,0)$-submodule of $\mathcal S_{sf}$  $($resp.$,$
$\mathcal S_{f})$ has the form $\mathcal S_{spf}$ or $\mathcal S_{tf}$ $($resp.$,$ $\mathcal S_{pf},$
or $\mathcal S_s$ in case $f=1$ and $h(0)=0)$ for some irreducible polynomial $p\in\C[t]$ $($i.e.$,$ $p=p_0t+p_1$ is a linear polynomial for some
$p_0\in\C^*,\,p_1\in\C)$.

\item[\rm (2)] Let $f(s,t)=\sum_{i=0}^nf_i(t)s^i\in\C[s,t]$ and assume $g(t)\in\C[t]$ is a monic polynomial.  Then  $\mathcal S_{f(s,t)}=\mathcal S_{sg}$   if and only if
\begin{eqnarray*}
\begin{cases}
(f_0, f_1,\ldots, f_n)=g(t)
\quad  \mbox{and}\quad  tg(t)\mid f_0(t)
 &\mbox{if $h(0)=0$},\\[4pt] (f_0, f_1,\ldots, f_n)=g(t) &\mbox{else},
\end{cases}
\end{eqnarray*}
where $(f_0,f_1,...,f_n)$ denotes the greatest common divisor of $f_0(t),f_1(t),...,f_n(t),$
and $\mathcal S_{f(s,t)}=\mathcal S_g$ if and only if
\begin{eqnarray*}
\begin{cases}
(f_0, tf_1,\ldots, tf_n)=g(t)\quad  \mbox{and}\quad  g(t)\mid f_i(t)\mbox{ for }i\geq1 &\mbox{if $h(0)=0$},\\[4pt]
(f_0, f_1,\ldots, f_n)=g(t) &\mbox{else}.
\end{cases}
\end{eqnarray*}
\end{itemize}
\end{prop}
\begin{proof}
(1) We only focus on the case $\mathcal S_{sf}$, the other case can be treated similarly. Assume $M$ is a maximal submodule of $\mathcal S_{sf}.$ Then being a submodule of $\Phi(\lambda, \alpha, h)$, by Theorem \ref{theo-1(b=0)}, we obtain that
either $M=\mathcal S_{sg}$  or $M=\mathcal S_g$ for some $g(t)\in\C[t]$.
Since $M\subset\mathcal S_{sf}$, we have $sg\in\mathcal S_{sf}$ or $g\in\mathcal S_{sf}$.
In any case using the fact that $\C[s,t]$ is a unique factorization domain, we deduce  that  $f(t)\mid g(t)$, i.e.,  $g(t)=p(t)f(t)$ for some $p(t)\in\C[t]$. Then $p(t)$ must be irreducible otherwise any nontrivial factor $q(t)$ of $p(t)$ would give rise to  a  submodule  $\mathcal S_{sqf}$  (resp., $\mathcal S_{qf})$ larger than $\mathcal S_{sg}$ (resp., $\mathcal S_{g}$). Note that if $M=\mathcal S_{g}=\mathcal S_{pf}\subseteq \mathcal S_{sf}$ then we must have $pf\mid  tf$, i.e., $p\mid t$, thus $p=t$ and $M=\mathcal S_{tf}.$

(2)  Note fron Lemma \ref{lemm-2(b=0)} that in case $h(0)=0$, we have
\begin{eqnarray*}&\!\!\!&\mathcal S_{f(s,t)}=\mathcal S_{sg}\\ &\!\!\!&\Longleftrightarrow
 \mbox{$\sum\limits_{i=1}^n$}\Big(\C[s,t]sf_i(t)+\C[s,t]tf_i(t)\Big)+\C[s,t]f_0(t)=\C[s,t]sg(t)+\C[s,t]tg(t)\\
&\!\!\!&\Longleftrightarrow
\exists\, q_i(t)\in\C[t]\mbox{ such that } sg(t)=\mbox{$\sum\limits_{i=0}^n$}sq_i(t)f_i(t),\
tg(t)\mid f_0(t),\ tg(t)\mid tf_i(t)\mbox{ for } i\geq1,\\
&\!\!\!&\Longleftrightarrow
(f_0, f_1,\ldots, f_n)=g(t)\mbox{ and } tg(t)\mid f_0(t),
\end{eqnarray*}
and in case $h(0)\ne0$, we have
\begin{eqnarray*}\mathcal S_{f(s,t)}=\mathcal S_{sg} &\!\!\!\Longleftrightarrow\!\!\!&
\mbox{$\sum\limits_{i=1}^n$}\C[s,t]f_i(t)+\C[s,t]f_0(t)=\C[s,t]g(t)\\
&\!\!\!\Longleftrightarrow\!\!\!&
\exists\, q_i(t)\in \C[t]\mbox{ such that } g(t)=\mbox{$\sum\limits_{i=0}^n$} q_i(t)f_i(t)
\mbox{ and } g(t)\mid f_i(t)\mbox{ for } i\geq0\\
&\!\!\!\Longleftrightarrow\!\!\!&
(f_0, f_1,\ldots, f_n)=g(t).
\end{eqnarray*}
In addition, in case $h(0)=0$ we have
\begin{eqnarray*}&&\mathcal S_{f(s,t)}=\mathcal S_{g}\\ &&\Longleftrightarrow
 \mbox{$\sum\limits_{i=1}^n$}\Big(\C[s,t]sf_i(t)+\C[s,t]tf_i(t)\Big)+\C[s,t]f_0(t) =\C[s,t]g(t)\\
&&\Longleftrightarrow
\exists\, q_i(t)\in\C[t]\mbox{ such that } g(t)= \mbox{$\sum\limits_{i=1}^n$}q_i(t)tf_i(t)+q_0(t)f_0(t)
 \mbox{ and } g(t)\mid f_i(t)\mbox{ for } i\geq0\\
&&\Longleftrightarrow
(f_0, tf_1,\ldots, tf_n)=g(t)\mbox{ and } g(t)\mid f_i(t)\mbox{ for }i\geq1,
\end{eqnarray*}
and in case $h(0)\ne0$ we have
\begin{eqnarray*}\mathcal S_{f(s,t)}=\mathcal S_{g}&\!\!\!\Longleftrightarrow\!\!\!&
\mbox{$\sum\limits_{i=0}^n$}\C[s,t]f_i(t)=\C[s,t]g(t)\\
&\!\!\!\Longleftrightarrow\!\!\!&
\exists\, q_i(t)\in \C[t]\mbox{ such that } g(t)=\mbox{$\sum\limits_{i=0}^n$} q_i(t)f_i(t)
\mbox{ and } g(t)\mid f_i(t)\mbox{ for } i\geq0\\
&\!\!\!\Longleftrightarrow\!\!\!&
(f_0, f_1,\ldots, f_n)=g(t).
\end{eqnarray*}
This completes the proof.
 \end{proof}

As a consequence of Proposition \ref{prop-1(b=0)} we have the following result.
\begin{coro}
Every $\V(0,0)$-submodule of $\Phi(\lambda,\alpha, h)$ has a composition series.
\end{coro}

\subsection{The $\V(0,1)$-submodules of  $\Phi(\lambda,\alpha, h)$  and $\Theta(\lambda,h)$}

Let us first determine  the $\V(0,1)$-submodules $\mathcal S_{f(s,t)}$ of $\Theta(\lambda,h)$.

\begin{lemm}\label{lemm1-(b=1)}Let  $1\le m\in\Z_+, n\in\Z_+,$ $p(t)\in\C[t]$ and $f(s,t)=\sum_{i=0}^ms^if_i(t)\in\C[s,t]$ with $f_i(t)\in\C[t]$ for $i=0,...,m$. Then
\begin{eqnarray}
\label{lemm1-(b=1)(1)}
&\!\!\!\!\!\!\!\!\!&
\mathcal S_{s^np(t)}= \C[s,t]s^{1-\delta_{n,0}}p(t)+\C[s,t]h(t)p(t),\\
\label{lemm1-(b=1)(3)}
&\!\!\!\!\!\!\!\!\!&
\mathcal S_{f(s,t)}=\mbox{$\sum\limits_{i=1}^{m}$}\mathcal S_{sf_i(t)}+\mathcal S_{f_0(t)}. \end{eqnarray}
\end{lemm}

\begin{proof}Obviously the right-hand side of \eqref{lemm1-(b=1)(1)} is a $\V(0,1)$-submodule containing $s^np(t).$ Thus
it suffices to show that $\mathcal S_{s^np(t)}$ contains $ \C[s,t]s^{1-\delta_{n,0}}p(t)+\C[s,t]h(t)p(t)$.
Assume $n\geq1$ (the case with $n=0$ is trivial). By Lemma \ref{lemm-1}, we have
\begin{eqnarray}
\label{Anooo111}
&\!\!\!\!\!\!&\mathcal S_{s^np(t)}\ni S_\Theta^ns^np(t)=\left(\frac{s}{n!}\partial_s^n-\frac{h(t)}{(n-1)!}\partial_s^{n-1}\right)s^np(t)=sp(t)-nsh(t)p(t),
\nonumber\\
&\!\!\!\!\!\!\!\!&\mathcal S_{s^np(t)}\ni -S_\Theta^{n+1}s^np(t)=\frac{1}{n!}\partial_s^nh(t) s^np(t)=h(t)p(t).
\end{eqnarray}Thus $sp(t)\in \mathcal S_{s^np(t)},$  which together with \eqref{Anooo111} implies \eqref{lemm1-(b=1)(1)}.

 Let $M$ be any $\V(0,1)$-submodule of $\Theta(\lambda,h)$ containing $f(s,t)$. For the equality in \eqref{lemm1-(b=1)(3)}, it suffices to show that
  $f_0(t), sf_1(t),...,sf_{m}(t)\in M.$    Note by  Lemma \ref{lemm-1} again that
  \begin{eqnarray} \label{a-d-d1}&& M\ni S_\Theta^{0}\big(S_\Theta^{m+1}f(s,t)\big)=S_\Theta^{0}\big(-h(t)f_m(t)\big)=-sh(t)f_m(t),\\
 \label{a-d-d2} && M\ni S_\Theta^{m}f(s,t)=sf_m(t)-h(t)\big(msf_m(t)+f_{m-1}\big),\\
  \label{a-d-d3}&&M\ni h(t)f(s,t).\end{eqnarray}
Assume first  $m=1$.  By \eqref{a-d-d2} and \eqref{a-d-d3},   $sf_1(t)\in M$ and then $f_0(t)=f(s,t)-sf_1(t)\in M,$ proving the case with $m=1$. Assume $m\geq 2$. Then it follows from \eqref{a-d-d1} and \eqref{a-d-d2} that $sf_m(t)-h(t)f_{m-1}(t)\in M$. By the previous case, $sf_m(t)\in M$. Thus, $\sum_{i=0}^{m-1}s^if_i(t)\in M,$ which together with the inductive assumption gives that $f_0(t), sf_1(t),..., sf_{m-1}(t)\in M$.
\end{proof}

\begin{theo}\label{th-(b=1)}
The set $\{\mathcal S_{f(t)}+\mathcal S_{sg(t)}\mid f(t), g(t)\in\C[t]\}$ exhausts all $\V(0,1)$-submodules of $\Theta(\lambda,h)$.
\end{theo}

\begin{proof}
Let $M$ be a nonzero $\V(0,1)$-submodule of $\Theta(\lambda,h).$  Take $ f(t)\in M\cap \C[t]$ and $sg(t)\in M\cap s\C[t]$ such that deg$\,f(t)$ and deg$\,g(t)$ are minimal (note that here $f$ may be zero so that deg\,$f=\infty$).  Then $M\cap \C[t]\subseteq \mathcal S_{f(t)}$ by the choice of $f(t)$ and by \eqref{lemm1-(b=1)(1)}. Similarly, $M\cap s\C[t]\subseteq \mathcal S_{sg(t)}.$  These together with  \eqref{lemm1-(b=1)(3)} force that any element of $M$ must lie in $\mathcal S_{sg(t)}+\mathcal S_{f(t)}$. Thus, $M=\mathcal S_{sg(t)}+\mathcal S_{f(t)}$.
\end{proof}

Note from \eqref{operator-ST} that in the case $b=1$, we have
\begin{eqnarray*}
&&\dis T^j=\frac {t-\alpha}{j!}\partial_s^j,\nonumber\\
\label{sb=1}&&
S^j=\frac s{j!}\partial_s^j-\frac 1{(j-1)!}\partial_s^{j-1}\Big(G(t)+(t-\al)\partial_t\Big)\  {\rm with}\ G(t)=h(t)-\alpha g(t).
\end{eqnarray*}
In particular, $T^0=t-\alpha$  and $S^0=s.$ Thus, \begin{eqnarray}\label{4552-0}\C[s,t]f(s,t)=\C[s,t-\alpha]f(s,t)\subseteq \mathcal S_{f(s,t)}.\end{eqnarray}
Now we can determine   the $\V(0,1)$-submodule $\mathcal S_{f(s,t)}$ of $\Phi(\lambda,\alpha,h)$ as follows.
\begin{lemm}\label{lemm2-(b=1)}
\begin{itemize}
\item[\rm (1)] Let $p(t)=\sum_{i\geq n}a_i(t-\alpha)^i\in\C[t]$ for some  $a_i\in\C$ with $a_n\neq0$ and let $1\le m\in\Z_+.$ Then
\begin{eqnarray}\label{eqno-lemm11}
\mathcal S_{p(t)}=\mathcal S_{s^mp(t)}=\mathcal S_{(t-\alpha)^n}=\C[s,t](t-\alpha)^n.\end{eqnarray}

\item[\rm (2)] Let $f(s,t)=\sum_{i\ge p}\sum_{j\ge l}a_{ij}s^i(t-\alpha)^j\in\C[s,t] $ for some $a_{ij}\in\C$ with $a_{p l}\neq0.$ Then $\mathcal S_{f(s,t)}=\mathcal S_{(t-\alpha)^l}.$
\end{itemize}
\end{lemm}
\begin{proof}
(1) It is easy to check that $\C[s, t-\alpha](t-\alpha)^n$ is $S^j$ and $T^j$-invariant for any $j\in\Z_+$. Thus by Lemma \ref{lemm-1}, $\C[s, t-\alpha](t-\alpha)^n$ is a $\V(0,1)$-module. It   follows from this and \eqref{4552-0} that $\mathcal S_{(t-\alpha)^n}=\C[s,t](t-\alpha)^n$.

By Lemma \ref{lemm-1},  we have \begin{eqnarray}\label{4552-2}&&\mathcal S_{s^mp(t)}\ni T^m\big(s^mp(t)\big)=(t-\alpha)p(t) ,\\  &&\label{4552-3}\mathcal S_{s^mp(t)}\ni  -S^{m+1}\big(s^mp(t)\big)=\Big(\big(G(t)-G(\alpha)\big)+\big(G(\alpha)+(t-\alpha)\partial_t\big)\Big)p(t).\end{eqnarray}
Note that $\gamma(t):=\frac{G(t)-G(\alpha)}{t-\alpha}\in\C[t].$ Then by \eqref{4552-0} and \eqref{4552-2}, $$\big(G(t)-G(\alpha)\big)p(t)=\gamma(t)(t-\alpha)p(t)\in \mathcal S_{(t-\alpha)p(t)}\subseteq \mathcal S_{s^mp(t)},$$ which together with  \eqref{4552-3} implies that $\big(G(\alpha)+(t-\alpha)\partial_t\big)p(t)\in \mathcal S_{s^mp(t)}.$ Inductively, one has, for $k\geq1$,    $$\mbox{$\sum\limits_{i\geq n}$}\big(G(\alpha)+i\big)^ka_i(t-\alpha)^i= \big(G(\alpha)+(t-\alpha)\partial_t\big)^kp(t)\in \mathcal S_{s^mp(t)}.$$  Thus, $(t-\alpha)^n\in \mathcal S_{s^mp(t)}$ and therefore $\mathcal S_{(t-\alpha)^n}\subseteq S_{s^mp(t)}\subseteq S_{p(t)}\subseteq \mathcal S_{(t-\alpha)^n}$  by \eqref{4552-0}, proving
\eqref{eqno-lemm11}.

(2) Denote $q=\max\{i\mid a_{ij}\ne0$ for some $j\}$. The result follows from (1) if $q=p$. Assume $q\geq p+1$. Write $f(s,t)$ as $\sum_{i=p}^q s^if_i(t)$ with $f_i(t)=\sum_{j\ge l} a_{ij}(t-\alpha)^j$. Then it follows from the proof of the part ``\,$\mathcal S_{(t-\alpha)^n}\subseteq\mathcal S_{s^mp(t)}$\,'' in (1) that $s^qf_q(t)\in \mathcal S_{f(s,t)}$.    Inductively,  $s^if_i(t)\in \mathcal S_{f(s,t)}$ for $i=q-1,...,p$.  Then by (1), $\mathcal S_{(t-\alpha)^l}\subseteq \mathcal S_{f(s,t)}$. The converse inclusion  follows from \eqref{4552-0}.
\end{proof}

\begin{theo}\label{the0-(b=1)}
The set $\{\mathcal S_{(t-\alpha)^n}\mid n\in\Z_+\}$ exhausts all nonzero $\V(0,1)$-submodules of $\Phi(\lambda,\alpha,h)$.
\end{theo}
\begin{proof}
Let $M$ be  a nonzero $\V(0,1)$-submodule of $\Phi(\lambda,\alpha,h)$. Choose $n\in\Z_+$ minimal such that $(t-\alpha)^n\in M$. It suffices to show  for any $f(s,t)=\sum_{i=p}^q\sum_{j=l}^k a_{ij} s^i(t-\alpha)^j\in M$  with $a_{pl}\ne0$, we must have  $f(s,t)\in \mathcal S_{(t-\alpha)^n}$.  
Thanks to  Lemma \ref{lemm2-(b=1)}\,(2), we have $f(s,t)\in \mathcal S_{f(s,t)}=\mathcal S_{(t-\alpha)^l}.$ The minimality of $n$ and \eqref{4552-0} imply that  $S_{(t-\alpha)^l}\subseteq \mathcal S_{(t-\alpha)^n}$.  Thus,  $f(s,t)\in\mathcal S_{(t-\alpha)^n}$.
\end{proof}

\subsection{$\V(0,b)$-submodules for $b\ne0,1$}

Note that if $b=-1$,
then by  \eqref{operator-ST} we have
\begin{eqnarray}
&\!\!\!\!\!\!\!\!\!\!\!\!\!\!\!\!&T^j=\frac t{j!}\partial_s^j+\frac\alpha{(j-1)!}\partial_s^{j-1},\nonumber  \\
\label{sb=-1}&\!\!\!\!\!\!\!\!\!\!\!\!\!\!\!\!&
S^j=\frac s{j!}\partial_s^j-\frac 1{(j-1)!}\partial_s^{j-1}\Big(t\big(g(t)-\partial_t\big)+h(\alpha)\Big)-\frac{1}{(j-2)!}\partial_s^{j-2}\alpha\big(g(t)-\partial_t\big).
\end{eqnarray}

\begin{theo}\label{theo-(b=-1}

\begin{itemize}
\item[\rm (1) ] $\Phi(\lambda,\alpha,h)$ is an irreducible $\V(0,-1)$-module if and only if $\alpha\neq0.$

\item[\rm (2)]  The set $\{\mathcal S_{t^i}=t^i\C[s,t]\mid i\in\Z_+\}$ exhausts all nonzero $\V(0,-1)$-submodules of $\Phi(\lambda,0,h)$ and also exhausts all $\V(0,b)$-submodules of $\Phi(\lambda,\alpha,h)$ when $b\neq \pm1,0$.

\end{itemize}
\end{theo}

\begin{proof} Note that $t\C[s,t]$ is a proper $\V(0,-1)$-submodule of $\Phi(\lambda,0, h)$. Consider now $\alpha\neq0$. For any $f(s,t)=\sum_{i=0}^n s^if_i(t)\in\C[s,t]$ with $f_n(t)\neq0$, by Lemma \ref{lemm-1} we have $$\mathcal S_{f(s,t)}\ni -S^{n+2}(T^0)^mf(s,t)=\alpha \big(g(t)-\partial_t\big)t^mf_n(t),$$
which is nonzero for sufficiently large $m.$  In particular,  there exists  $0\neq\gamma(t)\in \C[t]\cap \mathcal S_{f(s,t)}.$ It follows from the facts ``\,$g(t)\gamma(t)=g(T^0)\gamma(t)\in \mathcal S_{f(s,t)}$ and  $\big(g(t)-\partial_t\big)\gamma(t)=-\frac{S^2}{\alpha}\gamma(t)\in \mathcal S_{f(s,t)}$\,'' that $\partial_t\gamma(t)\in\mathcal S_{f(s,t)}$. Thus, induction on $\deg\,\gamma$ gives that 
$1\in\mathcal S_{f(s,t)}.$ Then $s^it^j=(S^0)^i(T^0)^j1\in\mathcal S_{f(s,t)}$ for any $i,j\in\Z_+$ and therefore $\mathcal S_{f(s,t)}=\C[s,t]$. This shows that $\Phi(\lambda,\alpha,h)$ is irreducible when $\alpha\neq0$, proving (1).

If $(b, \alpha)=(-1,0)$ then by \eqref{sb=-1} we have \begin{eqnarray*}
T^j=\frac t{j!}\partial_s^j,\  \ \  \ S^j=\frac s{j!}\partial_s^j-\frac 1{(j-1)!}\partial_s^{j-1}\big(tg(t)+h(0)-t\partial_t\big),
\end{eqnarray*}
and if $b\neq \pm1,0$ then by \eqref{operator-ST}  we have  \begin{eqnarray*}
T^j=\frac {t}{j!}\partial_s^j,\  \ \ \ \  S^j=\frac s{j!}\partial_s^j-\frac 1{(j-1)!}\partial_s^{j-1}\big(h(t)+bt\partial_t\big)\end{eqnarray*} Then (2) follows from the proof of Lemma \ref{lemm2-(b=1)} and Theorem \ref{the0-(b=1)}.
\end{proof}

As a consequence of  Theorems \ref{theo-1(b=0)}, \ref{th-(b=1)}, \ref{the0-(b=1)} and \ref{theo-(b=-1} we have the following (see also \cite[Proposition 2.3\,(iii)]{HCS}).
\begin{coro}\label{coro-red}
For $\lambda\in\C^*, \alpha\in\C$ and $h(t)\in\C[t],$ the $\V(0,1)$-module $\Theta(\lambda,h)$ is reducible and the $\V(0,b)$-module $\Phi(\lambda,\alpha,h)$ is  irreducible if and only if   $\alpha\neq0$ and $b=-1.$
\end{coro}

\section{$\mathscr V\!$-submodules }
  In this section for any $f(s,t)\in\C[s,t]$ we use  $\Psi_{f(s,t)}$ to  denote the $\mathscr V\!$-submodule of $\Phi(\lambda,\alpha, h)$ or  $\Theta(\lambda,h)$ generated by  $f(s,t)$.

A pair $(n,i)\in\Z_+^2$ with $i\le n$ is called {\em minimal} with respect to $ h(t)$ if for any other pair $(m,j)\in\Z_+^2$ such that $j\le m$ and $m\big({\rm deg}\, h(t)-1\big)+j=n\big({\rm deg}\, h(t)-1\big)+i$ one has $n\le m$.

\begin{lemm}\label{lemm-2} Let $b=-1,$ $\alpha\neq0$ and $f:=f(t)\in\C[t]$. Then
\begin{itemize}
\item[\rm (1)] $\Psi_{f}=\sum_{m=0}^\infty \sum_{i=0}^m\C[s]t^i\big(g(t)-\partial_t\big)^m f.$

\item[\rm (2)] $\Psi_{s^nf}=\C[s]sf+\Psi_{(g(t)-\partial_t)f}+\Psi_{(t(g(t)-\partial_t)+h(\alpha))f}$ for all $n\geq1$.

 \item[\rm (3)]  $\Psi_{f}$  has a basis $\mathcal B=\{s^lt^i(g(t)-\partial_t)^nf\mid l\in\Z_+, i\leq n,  (n,i)\in\Z_+^2\ is\ minimal\};$
  furthermore$,$ if $\deg\,h\ge2$ then the subspace  spanned by $\mathcal B-\{f\}$ is
 a maximal submodule of $\Psi_{f}$.
 In particular$,$ $\Psi_{f}$ is reducible if $\deg\, h\geq2$.
\end{itemize}

\end{lemm}
\begin{proof} Clearly, $\big(g(t)-\partial_t\big)t=t\big(g(t)-\partial_t\big)-1$, which holds on $\C[s,t]$. Note by Lemma \ref{lemm-1} and \eqref{sb=-1} that elements in $\Psi_{f}\cap \C[t]$ can  be  obtained by repeatedly applying $g(t)-\partial_t$ or $t\big(g(t)-\partial_t\big)$ to $f$.   Thus, $\Psi_{f}\cap \C[t]=\sum_{m=0}^\infty \sum_{i=0}^m\C t^i\big(g(t)-\partial_t\big)^m f.$ It is easy to check that $\sum_{m=0}^\infty \sum_{i=0}^m\C[s]t^i\big(g(t)-\partial_t\big)^m f$ is $S^j$-invariant for any $j\in\Z_+$ (cf. \eqref{sb=-1}),   and therefore is a $\mathscr V\!$-submodule containing $f$ by Lemma \ref{lemm-1}. It follows from these that  (1) holds.

Set $F=g(t)-\partial_t$. By Lemma  \ref{lemm-1}, \begin{eqnarray*}
 & &\Psi_{s^nf} \ni S^{n+2}(s^nf) =-\alpha Ff,\\
&&\Psi_{s^nf}\ni S^{n+1}(s^nf)=-\big(tF+h(\alpha)\big)f-n\alpha S^0(Ff),\\
&&\Psi_{s^nf}\ni S^{n}(s^nf)= sf-nS^0\Big(\big(tF+h(\alpha)\big)f\Big)-\binom{n}{2}\alpha (S^0)^2(Ff).
\end{eqnarray*}It follows from these that $Ff, \big(tF+h(\alpha)\big)f, sf\in \Psi_{s^nf}$ for all $n\geq1$.  In particular, $$\C[s]sf+\Psi_{Ff}+\Psi_{(tF+h(\alpha))f}\subseteq \Psi_{sf}\subseteq \Psi_{s^nf}.$$
Then (2) follows from this and the observation  that $\C[s]sf+\Psi_{Ff}+\Psi_{(tF+h(\alpha))f}$ is a $\mathscr V\!$-submodule  containing $sf$.

Assume  $2\le k:= {\rm deg}\, h(t)<\infty$. To show that $\mathcal B$ is a basis of $\Psi_{f},$ it is sufficient to show  that $\mathfrak M=\{t^i(g(t)-\partial_t)^nf\mid  i\leq n,  (n,i)\in\Z_+^2\ {\rm is\ minimal}\}$ is a basis of $$\Psi_{f}\cap \C[t]=\mbox{$\sum\limits_{m=0}^\infty \sum\limits_{i=0}^m$}\C t^i(g(t)-\partial_t)^m f.$$ The set $\mathfrak M$ is linearly independent, since elements in $\mathfrak M$ have distinct degrees.  Let $n\in\Z$. Note  that $$\Psi\cap \C[t]_{n+{\rm deg}\, f}={\rm span}\{ t^i(g(t)-\partial_t)^m f\mid 0\le i\le m, m(k-1)+i\le n\},$$  where in general $\C[t]_m$ is the subspace of $\C[t]$ consisting of polynomials whose degree is less than $m$ (here we assume ${\rm deg}\,0=-\infty$).  Using induction on $n$, we can assume that each element in $\Psi\cap \C[t]_{n+{\rm deg}\, f}$ is a linear combination of $\mathfrak M$ (which is obvious if $n=-\deg\,f-1$).
We show that $\Psi\cap \C[t]_{n+1+{\rm deg}\, f}$ is a linear combination of $\mathfrak M$.
 Clearly,  $$\Psi\cap \C[t]_{n+1+{\rm deg}\, f}=\Psi\cap \C[t]_{n+{\rm deg}\, f}\oplus{\rm span}_\C\{t^i(g(t)-\partial_t)^sf\mid 0\le i\leq s, s(k-1)+i=n+1\}.$$ Let $(s_0,i_0)$ be minimal among the set $\{(s,i)\mid 0\le i\le s, s(k-1)+i=n+1\}$. Note that for any  $(s,i)\in\Z_+^2$ such that $i\le s$ and $s(k-1)+i=n+1$, there exists $\zeta\in\C$ such that $${\rm deg}\,\Big(t^i\big(g(t)-\partial_t\big)^s f-\zeta t^{i_0}\big(g(t)-\partial_t\big)^{s_0}f\Big)\leq n+{\rm deg}\, f$$ that is  $t^i\big(g(t)-\partial_t\big)^s f-\zeta t^{i_0}\big(g(t)-\partial_t\big)^{s_0}f\in \Psi\cap \C[t]_{n+{\rm deg}\, f}.$   Then by the inductive assumption,  $t^i\big(g(t)-\partial_t\big)^s f$ is a linear combination of $\mathfrak M$, as desired.

Finally, it is easy to see that the linear span of $\mathcal B-\{f\}$ is $S^j$-invariant for $j\in\Z_+$, and thus is a $\mathscr V\!$-submodule by Lemma \ref{lemm-1}. The maximality  is trivial.
\end{proof}

\begin{theo}\label{th-irr}
The $\mathscr V\!$-module $\Phi(\lambda, \alpha,h)$ is irreducible if and only if $b=-1,$ $\alpha\neq0$ and ${\rm deg}\, h(t)=1$.
\end{theo}
\begin{proof}By Corollary \ref{coro-red},
the fact ``\,$\Phi(\lambda,\alpha,h)$ is irreducible as a $\mathscr V\!$-module\,'' can only occur when  $b=-1$ and $\alpha\neq0$. Thus we can assume $b=-1,\,\alpha\ne0$.
If  $h(t)\in\C$, then  $g(t)=0$ by \eqref{g-hi-}, and we obtain from Lemma \ref{lemm-2}\,(1) that $\C[s]=\Psi_{1}$ is a proper $\mathscr V\!$-submodule of $\Phi(\lambda,\alpha,h)$.
If $\deg\,h\ge2$, we obtain from Lemma \ref{lemm-2}\,(3) that $\Phi(\lambda,\alpha,h)$ is reducible.

It remains to show that $\Phi(\lambda, \alpha,h)$ is irreducible if ${\rm deg}\, h(t)=1$. Let $W$ be a nonzero $\mathscr V\!$-submodule of $\Phi(\lambda,\alpha,h)$. Note from by \eqref{g-hi-} that $\eta:=g(t)\in \C^*$ and also $0\neq W\cap \C[t]$. Take $0\neq f(t)\in W\cap \C[t]$ such that ${\rm deg}\, f(t)$ is minimal. Then $f(t)\in\C$, since by Lemma \ref{lemm-2}\,(1), $\partial_t f(t)=-(\eta-\partial_t)f(t)+\eta f(t)\in W\cap \C[t].$   Now by Lemma \ref{lemm-2}\,(1) again, we have $$W\supseteq\Psi_{f(t)}=\mbox{$\sum\limits_{m=0}^\infty \sum\limits_{i=0}^m\C[s]t^i(g(t)-\partial_t)^m f(t)=\sum\limits_{m=0}^\infty \sum\limits_{i=0}^m$}\C[s]t^i=\C[s,t],$$which implies that $W=\Phi(\lambda,\alpha,h)$.
Thus,  $\Phi(\lambda,\alpha,h)$ is irreducible when ${\rm deg}\, h(t)=1$.
\end{proof}

\begin{remark}\rm We remark that for the case $b=-1$,  $\mathscr V\!$-submodules$,$ especially maximal $\mathscr V\!$-submodules of $\Phi(\lambda,\alpha,h)$ were studied  in {\rm \cite{CG1}}$,$ in particular
 Theorem \ref{th-irr} was also obtained there. However it may be worthwhile to mention that  Lemma \ref{lemm-2}\,(3) solves one of the problems listed there.
\end{remark}

Note from \eqref{operator-ST} that if $b\neq -1$ then $S^j$ and $S^j_\Theta$ have the same form:
\begin{eqnarray}\label{H-h}
\!\!\!\!\!\!\!\!\!\!\!\!&&S^j_{\Theta}=\frac s{j!}\partial_s^j-\frac 1{(j-1)!}\partial_s^{j-1}H(t)\ {\rm with}\ H(t)=h(t),\nonumber \\
\!\!\!\!\!\!\!\!\!\!\!\!&&S^j=\frac s{j!}\partial_s^j-\frac 1{(j-1)!}\partial_s^{j-1}H(t)\ {\rm with}\ H(t)=\Big(h(t)-\delta_{b,1}\alpha g(t)+(bt-\delta_{b,1}\alpha)\partial_t\Big).
\end{eqnarray} Here we view $H(t)$ as an operator on  $\C[s,t]$.  Following the proof of Lemma \ref{lemm1-(b=1)} and using Lemma \ref{lemm-1} we can determine cyclic $\mathscr V\!$-submodules of $\Phi(\lambda,\alpha, h)$ and  $\Theta(\lambda,h)$ as follows.

\begin{prop}\label{lemm-02vbnot-1} Assume that $b\neq-1$. Let  $1\le m\in\Z_+, n\in\Z_+,$ $p(t)\in\C[t]$ and $f(s,t)=\sum_{i=0}^ms^if_i(t)\in\C[s,t]$ with $f_i(t)\in\C[t]$ for $i=0,...,m$.  Then \begin{eqnarray*}
&\!\!\!\!\!\!\!\!\!&
\Psi_{s^np(t)}= \C[s]s^{1-\delta_{n,0}}p(t)+\mbox{$\sum\limits_{i=1}^\infty$}\C[s]\big(H(t)\big)^ip(t),\\
&\!\!\!\!\!\!\!\!\!&
\Psi_{f(s,t)}=\mbox{$\sum\limits_{i=1}^{m}$}\Psi_{sf_i(t)}+\Psi_{f_0(t)}. \end{eqnarray*}




\end{prop}

Now we can prove the following.

\begin{theo}\label{th-vir1}
Each $\mathscr V\!$-submodule of $\Phi(\lambda,\alpha,h)$ or $\Theta(\lambda,h)$ is finitely generated if and only if ${\rm deg}\, h(t)\geq1.$
\end{theo}

\begin{proof}
Clearly,  ${\rm deg}\, h(t)={\rm deg}\, H(t)$ by \eqref{H-h} (note that deg\,$\partial_t=-1$). Assume first  $h(t)\in\C$. Note from Proposition \ref{lemm-02vbnot-1}  that   each $\mathscr V\!$-submodule $\Psi_{f(s,t)}$  contains  polynomials in $A[t]$ of only finitely many different degrees on $t$ (where $A=\C[s]$). Thus,  $\Phi(\lambda,\alpha,h)$ or $\Theta(\lambda,h)$ can not be finitely generated.

Consider now $1\le k:={\rm deg}\, h(t)<\infty$.  Let $M$ be a nonzero $\mathscr V\!$-submodule of $\C[s,t]$. We need to show that $M$ is finitely generated.  Below we only consider the  case $b\neq-1$ as the similar arguments can be applied to the case $b=-1$ (or see \cite{CG1}).   Set $$I=\{{\rm deg}\, f\mid 0\neq f\in\C[t]\cap M\},\ J=\{{\rm deg\,} f\mid sf(t)\in s\C[t]\cap M, f(t)\neq0\}.$$ Note from Proposition \ref{lemm-02vbnot-1} that $I=\cup_{i=1}^n\{d_i+lk\mid l\in\Z_+\}$ for some $d_i\in I$. Choose $p_i\in \C[t]\cap M$ such that ${\rm deg}\, p_i=d_i$ for $i=1,2,\ldots,n$. Similarly, we can choose $sq_i(t)\in s\C[t]\cap M$ for $i=1,2,\ldots, m$ such that $J=\cup_{i=1}^m\{{\rm deg}\, q_i+lk\mid l\in\Z_+\}.$ Let $W$ be the $\mathscr V\!$-submodule of $M$ generated by   $\{p_i(t), sq_j(t)\mid i=1,2,\ldots,n, j=1,2,\ldots, m\}$.
Take any $0\neq f(t)\in \C[t]\cap M.$ Then by induction on ${\rm deg}\, f$ one can show that  $f(t)\in  \sum_{i=1}^n \sum_{j=0}^\infty\C H(t)^jp_i(t).$   Thus, $f(t)\in \sum_{i=1}^n \Psi_{p_i(t)}\subseteq W$  by Proposition  \ref{lemm-02vbnot-1}, that is, $\C[t]\cap M\subseteq W$.   Similarly, we have $s\C[t]\cap M\subseteq W$. These and Proposition \ref{lemm-02vbnot-1} imply that $M=W$, i.e., $M$ is  finitely generated.
\end{proof}


\section {New irreducible $\mathscr V\!$-modules}\setcounter{clai}{0}
Motivated by \cite{TZ1} we consider the tensor product $\mathscr V\!$-modules  $\bigotimes_{i=1}^n \Phi(\lambda_i,\alpha_i,h_i)\otimes V$ in this section, where $V$ is an irreducible $\mathscr V\!$-module satisfying the condition that there exists $R_V$ such that $L_k$ acts locally finite on $V$ for all $k\geq R_V$. We remark that the tensor product modules  were studied in \cite{GLW} for the case $n=1$ and  in   \cite{LGW} in some general cases.  

We first recall all various known irreducible non-weight $\mathscr V\!$-modules, and then show that modules  under our consideration are irreducible and also give the necessary and sufficient conditions under which two of  these modules are isomorphic. Finally we shall  compare  these modules with the known non-weight modules.

Recall from \cite{LZ2} that for $\lambda\in\C^*$ and $b\in\C$, one can define a $\mathscr V\!$-module structure on $\C[t]$, which is denoted as  $\Omega(\lambda, b)$, as follows:
 for $i\in\Z_+, n\in\Z$,  \begin{equation}\label{lll-dda} L_n t^i=\lambda^n(t-n)^i\Big(t+n(b-1)\Big),\ \ \ \ C t^i=0.\end{equation}
Note that $\Omega(\lambda, b)$ is irreducible if and only if $b\neq1$.

Let $\mathscr V\!_+$ denote the subalgebra of $\mathscr V\!$ spanned  by $L_i$ for $i\in\Z_+$. For any $\mathbbm c\in\C$ and  $\mathscr V\!_+$-module $N$, form the induced $\mathscr V\!$-module ${\rm Ind}(N):=U(\mathscr V\!)\otimes_{U(\mathscr V\!_+)} N$. Denote $${\rm Ind}_\mathbbm c(N)={\rm Ind}(N)/(C-\mathbbm c){\rm Ind}(N),$$ where $U(\mathscr V\!)$ is the universal enveloping algebra of $\mathscr V\!$. Recall  from \cite{LZ2} that
for any $b^\prime\in\C$ and any irreducible $\C[t^{\pm1}, t\frac{d}{dt}]$-module $A$, there is a $\mathscr V\!$-module   on the vector space $A$ such that for $n\in\Z, w\in A$,   $$ L_n w=\big(t^n(t\frac{d}{dt})+nb^\prime t^n\big)w,\ \ \ \ C w=0,$$ which is denoted by $A_{b^\prime}.$   And $A_{b^\prime}$ is an irreducible $\mathscr V\!$-module if and only if one of the following conditions holds: (1) $b^\prime\neq 0\ \mbox{or}\ 1$; (2) $b^\prime=1$ and $t\frac{d}{dt} A =A$; (3) $b^\prime=0$ and $A$ is not isomorphic to the natural $\C[t^{\pm1}, t\frac{d}{dt}]$-module $\C[t,t^{-1}]$.

Let $r\in\Z_+$ and let $\mathscr V\!_r$ be the ideal of $\mathscr V\!_+$ spanned by $\{L_i\mid i> r\}$. Denote $\mathfrak a_r=\mathscr V\!_+/\mathscr V\!_r $. Let $M$ be an irreducible $\mathfrak a_r$-module such that $\bar L_r$ is injective on $M$, where $\bar L_i=L_i+\mathscr V\!_r$. Recall from
\cite{LLZ} that for any $\gamma(t)\in\C[t]\setminus\C$, the linear tensor product, denoted as $\mathcal N(M,\gamma(t)),$ of an $\mathfrak a_r$-module $M$ with the  Laurent polynomial ring $\C[t^{\pm1}]$, carries the structure of an irreducible $\mathscr V\!$-module with the action of $\mathscr V\!$ on $\mathcal N(M,\gamma(t))$ being defined, for  $k,n\in\Z, v\in M$, by  \begin{equation}\label{mag} L_n(v\otimes t^k)=\Big(kv+\mbox{$\sum\limits_{i=0}^r$}\frac{n^{i+1}}{(i+1)!} \bar L_i v\Big)\otimes t^{n+k}+v\otimes t^{n+k}\gamma(t),\ \ \ \ C(v\otimes t^k)=0.\end{equation} We call a $\mathscr V\!_r$-module $V$  {\em locally finite} if for any $v\in V$ the dimension of $U(\mathscr V\!_r)v$ is finite and call $V$  {\em locally nilpotent} if for any $v\in V$ there exists $n\in\Z_+$ such that $L_{i_1}L_{i_2}\cdots L_{i_n} v=0$ for any $L_{i_j}\in \mathscr V\!_r$.

We also need to recall irreducible highest weight modules over $\mathscr V\!$ (cf. \cite{KR}).  For any $\mathbbm c,\mathbbm h\in\C,$ let $I(\mathbbm c,\mathbbm h)$ be the left idea of $U(\mathscr V\!)$  generated by the set $\{L_0-\mathbbm h, C-\mathbbm c, L_i\mid i\geq1\}$. Form the quotient $M(\mathbbm c,\mathbbm h):= U({\mathscr V\!})/ I(\mathbbm c,\mathbbm h),$ which is called the Verma module with highest weight $(\mathbbm c,\mathbbm h).$ Let $W(\mathbbm c,\mathbbm h)$ be the unique maximal proper submodule of $M(\mathbbm c, \mathbbm h)$ and  $L(\mathbbm c,\mathbbm h)$ denote the irreducible highest weight module $M(\mathbbm c,\mathbbm h)/W(\mathbbm c,\mathbbm h)$ with highest weight $(\mathbbm c,\mathbbm h)$.

\begin{theo}{\rm  \cite[Theorems 1, 2 and Proposition 4]{MZ}}\label{th-cited} Let $N$ be an irreducible $\mathscr V\!_+$-module and $V$  an irreducible $\mathscr V\!$-module.
\begin{itemize}
\item[\rm (1)] If there exists $k\in\Z_+\setminus\{0\}$ satisfying the following two conditions:
\begin{itemize}
\item[{\rm (a)}] $L_k$ acts injectively on $N;$

\item[{\rm (b)}] $L_i N=0$ for all $i>k$.
\end{itemize}
\noindent Then for any $\mathbbm c\in\C$ the $\mathscr V\!$-module ${\rm Ind}_\mathbbm c(N)$ is irreducible.

\item[\rm (2)]  The following conditions are equivalent$:$
\begin{itemize}
\item[{\rm (i)}] There exists $l\in\Z_+$ such that each $L_m$ acts locally finite on $V$ for any $m\geq l;$
\item[{\rm(ii)}] There exists  $k\in\Z_+$ such that $V$ is a locally finite $\mathscr V\!_k$-module$;$
\item[{\rm (iii)}] There exists  $n\in\Z_+$ such that $V$ is a locally nilpotent $\mathscr V\!_n$-module$;$
\item[{\rm (iv)}] $V\cong L(\mathbbm c,\mathbbm h)$ or ${\rm Ind}_\mathbbm c(W)$ for some $\mathbbm c,\mathbbm h\in\C,$ and for some irreducible $\mathscr V\!_+$-module $W$ satisfying the conditions  in {\rm (1)}.
\end{itemize}
\end{itemize}
\end{theo}
\begin{nota}\label{rem-s5}\rm
\begin{itemize}
\item[(1)]
Assume $V$ is an irreducible $\mathscr V\!$-module for which there exists $R_V\in\Z_+$ such that $L_i$ for all $i>R_V$ are locally finite on $V$. By Theorem \ref{th-cited}\,(2)  for any $v\in V,$ there exists $K\in\Z_+$ such that $L_iv=0$ for all $i\geq K$. The minimal  such $K$ for $v$ is denoted by $K_v$.
\item[(2)] For any $1\le n\in\Z_+$, denote ${\bf 1}_n=\underbrace{1\otimes 1\otimes \cdots \otimes 1}_n.$
\end{itemize}
\end{nota}

The assertion in \cite[Proposition 7]{LZ1} can be 
generalized as follows.
\begin{lemm}\label{LZ1-cited-s5}
Let $P$ be a vector $\C$-space and $P_1$ a subspace of $P.$ Suppose  $\lambda_1,\lambda_2,\ldots,\lambda_s\in\C^*$ are pairwise distinct$,$ and assume $v_{ij}\in P$ and $f_{ij}(t)\in\C[t]$ with ${\rm deg}\, f_{ij}(t)=j$ for $i=1,2,\ldots, s,\, j=0,1,2,\ldots, k.$ If $$\mbox{$\sum\limits_{i=1}^s\sum\limits_{j=0}^k$}\lambda_i^mf_{i,j}(m)v_{i,j}\in P_1\quad {\it for} \ L< m\in\Z
,$$
where $L$ is any fixed integer or $-\infty,$ then $v_{ij}\in P_1$ for all possible $i,j$.
\end{lemm}

Throughout the rest of  this section,  we always assume all $\mathscr V\!$-modules of the form $\Phi(\lambda,\alpha,h)$  are  irreducible (cf.~Theorem \ref{th-irr}).

\begin{theo}\label{th1-s5}
Let $ \lambda_i,\alpha_i\in\C^*,$ $h_i(t)\in\C[t]$ for $i=1,2,\ldots, n$ such that $\lambda_i$ are pairwise distinct and  ${\rm deg}\, h_i(t)=1$ for all $i.$ Let $V$  be an irreducible $\mathscr V\!$-module for which there exists $R_V\in\Z_+$ such that  $L_k$ for $k\geq R_V$ are locally finite on $V.$ Then
 \begin{itemize} \item[\rm(i)]  $\bigotimes_{i=1}^n\Phi(\lambda_i,\alpha_i,h_i)$ is  irreducible as a  $\mathscr V\!_r$-module for any  $r\in\Z_+;$
 \item[\rm(ii)]the tensor module $\bigotimes_{i=1}^n\Phi(\lambda_i,\alpha_i,h_i)\otimes V$ is an irreducible $\mathscr V\!$-module$;$
\item[\rm(iii)]in particular$,$ $\bigotimes_{i=1}^n\Phi(\lambda_i,\alpha_i,h_i)$ is an irreducible $\mathscr V\!$-module$.$

\end{itemize}\end{theo}
\begin{proof} Set $T=\bigotimes_{i=1}^n\Phi(\lambda_i,\alpha_i,h_i)$ and $\eta_i=\frac{h_i(t)-h_i(\alpha_i)}{t-\alpha_i}\in \C^*$ for $i=1,2,\ldots,n$. Take any $r\in\Z_+$ and $0\neq w\in T$.   Let $T_w$ be the $\mathscr V\!_r$-submodule of $T$ generated by $w$. Write $$w=\mbox{$\sum\limits_{I=(i_1,\ldots,i_n)\in \Gamma_1,\, J=(j_1,\ldots, j_n)\in\Gamma_2}$} a_{I J} s^{i_1}t^{j_1}\otimes s^{i_2}t^{j_2}\otimes\cdots\otimes s^{i_n}t^{j_n},$$
for some finite subsets $\Gamma_1,\Gamma_2\subset\Z_+^n$, $a_{IJ}\in\C^*$  for all $I\in\Gamma_1, J\in\Gamma_2$ such that all terms in the sum are linearly independent.
\begin{clai*}\label{c-l-m-01}
For any $1\le k\le n,\,j\in\Z_+,$ we have $$w_{kj}:=\mbox{$\sum\limits_{I, J} $} a_{I J l}s^{i_1}t^{j_1}\otimes \cdots\otimes S^js^{i_k}t^{j_k}\otimes \cdots\otimes s^{i_n}t^{j_n}\in T_w.$$
\end{clai*}
Then \eqref{ammel}, for any $m\geq r$,
\begin{eqnarray*}
T_w\ni L_mu&=&\mbox{$\sum\limits_{I, J}\sum\limits_{k=1}^n$}  a_{I J}s^{i_1}t^{j_1}\otimes \cdots\otimes \mbox{$\sum\limits_{j=0}^\infty$}\lambda_k^m (-m)^jS^js^{i_k}t^{j_k}\otimes \cdots\otimes s^{i_n}t^{j_n}\\
&=&\mbox{$\sum\limits_{k=1}^n\sum\limits_{j=0}^\infty$}\lambda_k^m(-m)^j w_{kj}.
\end{eqnarray*} Now applying Lemma \ref{LZ1-cited-s5} we see that $w_{k,j}\in T_w$ for any $k,j$, proving the claim.

 For a fixed  $k$, let $(i_k^0, j_k^0)$ be maximal in  the alphabetical order on $\Z_+^2$ among all $(i_k, j_k)$, where $i_k$ (resp. $j_k$) is  the $k$-th  component of $I$ (resp. $J$) and $I, J$ range respectively over $\Gamma_1, \Gamma_2$.     Then $$w_{k, i^0_k+2}=-\alpha_k\sum_{ I=(i_1, \ldots, i_k=i_k^0,\ldots, i_n)
 , J}  a_{I J }s^{i_1}t^{j_1}\otimes \cdots\otimes (\eta_k-\partial_t)t^{j_k}\otimes \cdots\otimes s^{i_n}t^{j_n},$$ and applying $L_m$ to $w_{k, i^0_k+2}$ one has   $$\mbox{$\sum\limits_{ I=(i_1, \ldots, i_k=i_k^0,\ldots, i_n)
 ,J}$}  a_{I J}s^{i_1}t^{j_1}\otimes \cdots\otimes (\eta_k-\partial_t)^2t^{j_k}\otimes \cdots\otimes s^{i_n}t^{j_n}\in T_w.$$ It follows from these that $$\mbox{$\sum\limits_{I=(i_1, \ldots, i_k=i_k^0,\ldots, i_n)
 ,J}$}  a_{I J }s^{i_1}t^{j_1}\otimes \cdots\otimes (\eta_k-\partial_t)j_kt^{j_k-1}\otimes \cdots\otimes s^{i_n}t^{j_n}\in T_w.$$ Repeating the above procedure finitely many times (if necessary) gives
$$0\neq \mbox{$\sum\limits_{ I=(i_1, \cdots, i_k=i_k^0,\ldots, i_n)
, J=(j_1,..., j_k=j_k^0,..., j_n)
}$}  a_{I J }s^{i_1}t^{j_1}\otimes \cdots\otimes 1\otimes \cdots\otimes s^{i_n}t^{j_n}\in T_w.$$
Doing this for the other factors yields  ${\bf 1}_n\in T_w$ (cf. Notation \ref{rem-s5}).
Note from the claim above that all elements obtaining by repeatedly applying $S^j\ (j\in\Z_+)$ to any factor  of ${\bf 1}_n$ lie in $T_w$.  Then Lemma \ref{lemm-1} and  the irreducibility of $\Phi(\lambda_i,\alpha_i,h_i)$, $T\subseteq T_w$. That is, $T$ is irreducible, i.e., we have (i).

Let $M$ be any nonzero $\mathscr V\!$-submodule of $T\otimes V.$ Take any $0\neq u=\sum_{i=1}^m w_i\otimes v_i\in M$ with $w_i\in T$ and $ v_i\in V$ such that $\{w_i\otimes v_i\mid i=1,2,\ldots, m\}$ is linearly independent. Let $K={\rm max}\{K_{v_i}\mid i=1,2,\ldots, m\}$ (cf. Notation \ref{rem-s5}).  Viewing $M$ as a $\mathscr V\!_K$-module and  following from the proof of (i) we  can first obtain ${\bf 1}_n\otimes v\in M$ for some $0\neq v\in V$ and then  $T\otimes \C v\subseteq M$. Thus, $T\otimes V\subseteq M$, as $V$ is irreducible. That is, $T\otimes V$ is an irreducible $\mathscr V\!$-module, i.e., we have (ii). (iii) follows from (i),  or (ii) by taking $V=L(0,0)=\C$ (the trivial $\mathscr V\!$-module).
\end{proof}

\begin{theo}\label{th2-s5}Let $j=1,2$. Let $ \lambda^{(j)}_i,\alpha^{(j)}_i\in\C^*,$ $h^{(j)}_i(t)\in\C[t]$ such that  $\lambda_i^{(j)}$ are pairwise distinct for the fixed $j$ and   ${\rm deg}\, h^{(j)}_i(t)=1.$ Let $V_j$ be an irreducible $\mathscr V\!$-module for which there exists $R_j\in\Z_+$ such that all $L_k$ for $k\geq R_j$ are locally finite
on $V_j$ for 
$j=1,2$.  Then for any $m, n\in\Z_+\setminus\{0\},$
 $$\bigotimes_{i=1}^{m}\Phi(\lambda^{(1)}_i,\alpha^{(1)}_i,h^{(1)}_i)\otimes V_1\cong  \bigotimes_{i=1}^{n}\Phi(\lambda^{(2)}_i,\alpha^{(2)}_i,h^{(2)}_i)\otimes V_2$$ as $\mathscr V\!$-modules if and only if $$ m=n, \ \ \  V_1\cong V_2\  as\ \mathscr V\!\text-modules\ \ and \ \ (\lambda^{(1)}_i,\alpha^{(1)}_i\eta^{(1)}_i)=(\lambda^{(2)}_{\sigma ( i)},\alpha^{(2)}_{\sigma i}\eta^{(2)}_{\sigma i})$$  for some $\sigma\in S_{m}$\ $($the $m$-th symmetric group$),$  where $\eta^{(j)}_i=\frac{h^{(j)}_i(t)-h^{(j)}_i(\alpha_i^{(j)})}{t-\alpha_i^{(j)}}\in \C^*$.  In particular$,$  $\Phi(\lambda,\alpha,h)\cong \Phi(\lambda^\prime, \alpha, h^\prime)$ as $\mathscr V\!$-modules if and only if $(\lambda, \alpha h)=(\lambda^\prime, \alpha^\prime h^\prime)$.
\end{theo}

\begin{proof}The second statement follows from a special case of the first one: $$m=n=1\ {\rm and}\  V_1=V_2=L(0,0)=\C.$$
So it suffices to show the first statement.    Suppose  $$\phi: \bigotimes_{i=1}^{m}\Phi(\lambda^{(1)}_i,\alpha^{(1)}_i,h^{(1)}_i)\otimes V_1\rightarrow \bigotimes_{i=1}^{n}\Phi(\lambda^{(2)}_i,\alpha^{(2)}_i,h^{(2)}_i)\otimes V_2$$ is an isomorphism of $\mathscr V\!$-modules.

\setcounter{clai}{0}\begin{clai}\label{cl-s5}We have
$m=n$ and there exist a $\sigma\in S_{m}$ and a linear map $\tau: V_1\rightarrow V_2$ such that $(\lambda^{(1)}_i,\alpha^{(1)}_i\eta^{(1)}_i)=(\lambda^{(2)}_{\sigma (i)}, \alpha^{(2)}_{\sigma (i)}\eta^{(2)}_{\sigma (i)})$ for all $1\le i\le m$  and $\phi({\bf 1}_{m}\otimes v)={\bf 1}_{m}\otimes \tau(v)$ for $v\in V_1$ $($cf. Notation \ref{rem-s5}$)$. 

\end{clai}
To prove the claim, take any fixed   $0\neq u\in V_1$ and write \begin{equation}\label{new-eq-a}\phi({\bf1}_m\otimes u)= \mbox{$\sum\limits_{l=1}^p$}f_{1l}\otimes f_{2l}\otimes \cdots\otimes f_{n l}\otimes u_l,\end{equation}for some $f_{li}\in \Phi(\lambda^{(2)}_i,\alpha^{(2)}_i,h^{(2)}_i)$ and $u_l\in V_2$.    Set $K={\rm max}\{K_u, K_{u_l}\mid 1\le l\le p\}.$ By \eqref{ammel}, for any $q\geq K$,
\begin{eqnarray*}
\!\!\!\!\!\!\!\!&&\mbox{$\sum\limits_{k=1}^{m}\sum\limits_{j=0}^2$}(\lambda_k^{(1)})^q(-q)^j \phi({\bf 1}_{k-1}\otimes S^j1\otimes 1\otimes \cdots\otimes 1\otimes u)=\phi\big(L_q({\bf1}_m\otimes u)\big)\\
\!\!\!\!\!\!\!\!&=&L_q\phi({\bf1}_m\otimes u)=\mbox{$\sum\limits_{l=1}^p\sum\limits_{k=1}^{n}\sum\limits_{j=0}^\infty$} (\lambda_k^{(2)})^q(-q)^j
f_{1l}\otimes \cdots\otimes f_{k-1,l}\otimes
S^jf_{kl}\otimes f_{{k+1}l}\otimes \cdots\otimes f_{n l}\otimes u_l.
\end{eqnarray*}
In view of Lemma \ref{LZ1-cited-s5}, we must have $m=n$ and furthermore we may assume $\lambda^{(1)}_i=\lambda^{(2)}_{ i}$ for all $1\le i\le m$ by {re-ordering} $\lambda^{(2)}_{ 1},\ldots, \lambda^{(2)}_{m}$ if necessary; moreover, $f_{kl}\in\C[t]$ for any $k,l$. In addition,  for $0\le j\le2$, we have $$\phi({\bf 1}_{k-1}\otimes S^j1\otimes 1\otimes\cdots\otimes 1\otimes u)=\mbox{$\sum\limits_{l=1}^p$}
f_{1l}\otimes\cdots\otimes f_{k-1,l}\otimes
S^jf_{kl}\otimes f_{{k+1}l}\otimes \cdots\otimes f_{m l}\otimes u_l.$$
  By this,  the fact  $``S^21=-\alpha_k^{(1)}\eta_k^{(1)}"$ 
  and  \eqref{new-eq-a}, we  have, 
\begin{equation}\label{XXXX}\mbox{$\!\!\!\!\!\sum\limits_{l=1}^p
f_{1l}\otimes\cdots\otimes f_{k-1,l}\otimes
S^2f_{kl}\otimes f_{{k+1}l}\otimes \cdots\otimes f_{m l}\otimes u_l=-\alpha_k^{(1)}\eta_k^{(1)}\sum\limits_{l=1}^pf_{1l}\otimes \cdots\otimes f_{m l}\otimes u_l.\!\!$}\end{equation}
To complete  the proof of the claim, it suffices to  show  $\alpha^{(1)}_k\eta^{(1)}_k=\alpha^{(2)}_k\eta^{(2)}_k$  and $f_{kl} \ (1\le l\le p)$ can be chosen to be  $\delta_{1l}$ for $1\le k\le m$. For convenience, we only show this for $k=1$. Note that  \eqref{new-eq-a} can be rewritten as   $$\phi({\bf1}_m\otimes u)=\mbox{$\sum\limits_{l=1}^{p^\prime}$}t^{d_l}\otimes w_l,$$
for some $w_l\in \bigotimes_{i=2}^{m}\Phi(\lambda^{(2)}_i,\alpha^{(2)}_i,h^{(2)}_i)\otimes V_2$ and some $p'\le p$ and some $d_l$ 
such that $d_1<d_2<\cdots< d_{p'}
$.
Then using this expression,    the special case  of \eqref{XXXX} with $k=1$ turns out to be \begin{equation}\label{yy-1=-}\mbox{$\sum\limits_{l=1}^{p^\prime}\alpha_1^{(2)}(\eta_1^{(2)}t^{d_l}-d_lt^{d_l-1})\otimes w_l=\alpha_1^{(1)}\eta_1^{(1)}\sum\limits_{l=1}^{p^\prime}t^{d_l}\otimes w_l.$}\end{equation}  We conclude that $d_1=0$ and $\alpha^{(1)}_1\eta^{(1)}_1=\alpha^{(2)}_1\eta^{(2)}_1$.  These together with \eqref{yy-1=-} imply 
\begin{equation*}\mbox{$\sum\limits_{l=2}^{p^\prime}\alpha_1^{(2)}(\eta_1^{(2)}t^{d_l}-d_lt^{d_l-1})\otimes w_l=\alpha_1^{(1)}\eta_1^{(1)}\sum\limits_{l=2}^{p^\prime}t^{d_l}\otimes w_l,$}\end{equation*}
 from which  we can similarly conclude $d_2=0,$  contradicting   $d_1<d_2$. Thus, $p'=1.$  To sum up, we have obtained $\alpha^{(1)}_1\eta^{(1)}_1=\alpha^{(2)}_1\eta^{(2)}_1$ and $\phi({\bf 1}_m\otimes u)=1\otimes w_1.$ This, in fact, is the case with $k=1$.  

\begin{clai}
The linear map $\tau: V_1\rightarrow V_2$ is an isomorphism of $\mathscr V\!$-modules.
\end{clai}  Take any $v\in V_1$.
Note that $\phi\big((\mathfrak b{\bf 1}_m)\otimes v\big)=(\mathfrak b{\bf 1}_m)\otimes \tau(v)$ for all $\mathfrak b\in U(\mathscr V\!_{K^\prime})\ ({\rm where}\ K^\prime={\rm max}\{K_v,K_{\tau(v)}\})$  and $\bigotimes_{i=1}^{m}\Phi(\lambda^{(1)}_i,\alpha^{(1)}_i,h^{(1)}_i)$ is an irreducible $\mathscr V\!_{K^\prime}$-module (cf. Theorem \ref{th1-s5}\,(i)), $\phi\big(({\mathfrak b}{\bf1}_m)\otimes v\big)=({\mathfrak b}{\bf1}_m)\otimes \tau(v)$ for all ${\mathfrak b}\in U(\mathscr V\!)$. In particular, $\phi\big((L_q{\bf1}_m)\otimes v\big)=(L_q{\bf1}_m)\otimes \tau(v)$  for all $q\in\Z$. It follows from this and $\phi\big(L_q({\bf1}_m\otimes v)\big)=L_q\big({\bf1}_m\otimes \tau(v)\big)$  we can easily deduce that $\tau(L_q v)=L_q\tau(v)$ for all $q\in\Z$. That is, $\tau$ is a homomorphism. Clearly, $\tau$ is a linear isomorphism. Thus, $\tau$ is an isomorphism of  $\mathscr V\!$-modules.

Conversely, suppose $m=n,$ $\tau^\prime : V_1\rightarrow V_2$ is an isomorphism of $\mathscr V\!$-modules,  and there exists $\sigma\in  S_m$ such that $(\lambda^{(1)}_i,\alpha^{(1)}_i\eta^{(1)}_i)=(\lambda^{(2)}_{\sigma ( i)},\alpha^{(2)}_{\sigma i}\eta^{(2)}_{\sigma i})$ for all $1\le i\le m.$ Then the linear map sending ${\bf 1}_m\otimes v$  to ${\bf 1}_m \otimes \tau^\prime(v)$ for $v\in V_1$  can be  uniquely extended to an isomorphism $\phi^\prime: \bigotimes_{i=1}^{m}\Phi(\lambda^{(1)}_i,\alpha^{(1)}_i,h^{(1)}_i)\otimes V_1\rightarrow \bigotimes_{i=1}^{m}\Phi(\lambda^{(2)}_i,\alpha^{(2)}_i,h^{(2)}_i)\otimes V_2$ of $\mathscr V\!$-modules.   This completes the  proof.
\end{proof}

\begin{theo}\label{th3-s5}Let $ \lambda_i,\alpha_i\in\C^*,$ $h_i(t)\in\C[t]$ such that the $\lambda_i$ are pairwise distinct and ${\rm deg}\, h_i(t)=1$  for $1\leq i\leq n$ and $V$  an irreducible $\mathscr V\!$-module for which there exists $R_V\in\Z_+$ such that all $L_k$ for $k\geq R_V$ are locally finite on $V$.
 Then $\bigotimes_{i=1}^n \Phi(\lambda_i,\alpha_i,h_i)\otimes V$ is not isomorphic to any of the following irreducible $\mathscr V\!$-modules$:$

   $$V^\prime, \mathcal N\big(M, \gamma(t)\big), A_{b^\prime} \  {\rm or}\  \bigotimes_{i=1}^m \Omega(\mu_i, b_i)\otimes V^\prime ,$$ where  $1\le m\in\Z_+, b^\prime, 1\neq b_i\in\C, \mu_i\in \C^*,$ $\gamma(t)\in \C[t]\setminus\C,$  $M$ is an irreducible $\mathfrak a_r$-module$,$ for $V^\prime$  there exists $R_{V^\prime}\in \Z_+$ such that $L_i$   for $i\geq R_{V^\prime}$ are all locally finite on $V^\prime$.
\end{theo}

\begin{proof} Denote $\bigotimes_{i=1}^n \Phi(\lambda_i,\alpha_i,h_i)\otimes V$ by $X$. The proof of $X\ncong \bigotimes_{i=1}^m \Omega(\mu_i, b_i)\otimes V^\prime,$ in fact,  follows  from that of Theorem \ref{th2-s5}. Here we give the detail. Suppose  $\phi: X\rightarrow \bigotimes_{i=1}^m \Omega(\mu_i, b_i)\otimes V^\prime$ is an isomorphism of $\mathscr V\!$-modules.  Take any    $0\neq u\in V$ and write \begin{equation}\label{new-eq-ab}\phi({\bf1}_n\otimes u)= \sum_{l\in I}f_{1l}\otimes f_{2l}\otimes \cdots\otimes f_{m l}\otimes u_l\quad ({\rm finite\ sum}).\end{equation}    Set $K={\rm max}\{K_u, K_{u_l}\mid l\in I\}.$ Then by \eqref{ammel} and \eqref{lll-dda}, for any $p\geq K$,
\begin{eqnarray*}
&&\sum_{k=1}^{n}\sum_{j=0}^2\lambda_k^p(-p)^j \phi({\bf 1}_{k-1}\otimes S^j1\otimes 1\otimes\cdots\otimes 1\otimes u)=\phi\big(L_p({\bf1}\otimes u)\big)\\
&=&L_p\phi({\bf1}\otimes u)=\sum_{l}\sum_{k=1}^{m}\underbrace{f_{1l}\otimes \cdots\otimes}_{k-1} \mu_k^p\big(t+p(b_k-1)\big)f_{kl}(t-p)\otimes f_{{k+1}l}\otimes \cdots\otimes f_{m l}\otimes u_l.
\end{eqnarray*} By Lemma \ref{LZ1-cited-s5},  we have $n=m$ and may assume $\lambda_i=\mu_i$ for $1\le i\leq n$; moreover, \begin{eqnarray}\label{*-0}
\sum_{j=0}^2(-p)^j \phi( S^j1\otimes 1\cdots\otimes 1\otimes u)=\sum_{l} \big(t+p(b_1-1)\big)f_{1l}(t-p)\otimes f_{{2}l}\otimes \cdots\otimes f_{m l}\otimes u_l.
\end{eqnarray}  Choose $f_{1,l}$  to be of the form $t^{d_l}$  such that $d_i\neq d_j$ for $i\neq j$  and write \begin{equation}\label{***-1}\phi({\bf 1}_n\otimes u)=\sum_{l=1}^kt^{d_l}\otimes w_l\end{equation}   for   some  $w_l\in \bigotimes_{i=2}^m \Omega(\mu_i, b_i)\otimes V^\prime.$
Using this expression,  \eqref{*-0} turns out to be
$$\sum_{j=0}^2(-p)^j \phi( S^j1\otimes 1\otimes\cdots\otimes 1\otimes u)=\sum_{l=1}^k \big(t+p(b_1-1)\big)(t-p)^{d_l}\otimes w_l.$$
By  Lemma \ref{LZ1-cited-s5}  and comparing the coefficients of $(-p)^2$,   we see that $(k,d_1)=(1,1)$  and  $\phi(S^21\otimes1\otimes \cdots\otimes 1\otimes u)=(1-b_1)1\otimes w_1.$      The former and \eqref{***-1} imply $\phi({\bf 1}\otimes u)=t\otimes w_1;$ the latter and $S^21\in\C^*$ imply $\phi({\bf 1}\otimes u)\in \C \otimes w_1.$  This is a contradiction.

Next we only prove $X\ncong \mathcal N\big(M, \gamma(t)\big)$ and the rest can be shown similarly. Suppose $\varphi: X\rightarrow \mathcal N\big(M, \gamma(t)\big)$ is an   isomorphism of $\mathscr V\!$-modules. Let $I=\{\lambda_i\mid i=1,2,\ldots, m\}\cup\{\lambda_0=1\}$. Take any nonzero $u\in V$.  Note that,  by \eqref{sb=-1},  $0\neq S^{r+1}s^{r+1}\otimes1\otimes\cdots\otimes 1\otimes u$ and also that, by \eqref{ammel} and \eqref{mag},      \begin{eqnarray*}0&=&\varphi\big(L_m(s^{r+1}\otimes1\otimes\cdots\otimes 1\otimes u)\big)-L_m\varphi(s^{r+1}\otimes1\otimes\cdots\otimes  1\otimes u)\\
&=&\lambda_1^m(-m)^{r+1}\varphi(S^{r+1}s^{r+1}\otimes1\otimes\cdots\otimes 1\otimes u)+\sum_{\lambda_i\in I}\sum_{j=0}^r \lambda_i^m(-m)^jv_{ij}\end{eqnarray*}  for some $v_{ij}\in  \mathcal N\big(M, \gamma(t)\big)$. Then by Lemma \ref{LZ1-cited-s5}, $\varphi(S^{r+1}s^{r+1}\otimes1\otimes\cdots\otimes 1\otimes u)=0$, contradicting the injectivity of $\varphi$. Thus, $X\ncong \mathcal N\big(M, \gamma(t)\big)$.
\end{proof}


\noindent{\bf \Large Acknowledgments}

This work was partially supported by NSFC grant no. 11971350.

\def\NOUSE#1{}\NOUSE

\cl{}  School of Mathematical Sciences, Tongji University, Shanghai
200092, China.

\vspace{.2cm}
 E-mail: jzhan@tongji.edu.cn

 \vspace{.1cm}
 E-mail: ycsu@tongji.edu.cn

\end{document}